\documentclass[12pt,a4paper,twoside]{article}
\usepackage{graphics}
\usepackage{amssymb}
\usepackage{amsmath}
\usepackage{mathrsfs}
\usepackage{makeidx}
\usepackage{color}
\usepackage{graphicx}
\usepackage{subfigure}
\usepackage{multicol}
\usepackage{multirow}
\usepackage{float}
\usepackage{booktabs}
\usepackage{epsfig}
\usepackage{multirow}
\usepackage{epstopdf}
\usepackage[colorlinks,
linkcolor=red,
anchorcolor=blue,
citecolor=blue
]{hyperref} 
\usepackage{caption}
\usepackage{cases}
\usepackage{algorithm,algorithmic}

\usepackage{latexsym, cite, bm, amsthm, amsmath}
\usepackage{enumerate}
\usepackage{marginnote}
\usepackage{epstopdf}
\usepackage{pifont}
\usepackage{cleveref}

\def \[{\begin{equation}}
	\def \]{\end{equation}}

\newtheorem{thm}{Theorem}[section]

\newtheorem{lem}{Lemma}[section]

\newtheorem{rem}{Remark}[section]

\newtheorem{exam}{Example}[section]

\numberwithin{equation}{section}

\title{\bf A generalization of the Gauss-Seidel iteration method for generalized absolute value equations}

\usepackage[marginal]{footmisc}
\usepackage{authblk}
\author[a]{Tingting Luo\thanks{Email address: 610592494@qq.com.}}
\author[a]{Jiayu Liu\thanks{Email address: 1977078576@qq.com.}}
\author[a]{Cairong Chen\thanks{Supported in part by the Fujian Alliance of Mathematics (2023SXLMQN03). Email address: cairongchen@fjnu.edu.cn.}}
\author[a]{Linjie Chen\thanks{Supported in part by the Fujian Social Science Foundation (No. FJ2024BF057), the Fujian Provincial Education Science Research Foundation (No. FJJKBK21-010) and the Teaching Reform Research project of Fujian Normal University (No. I202301062). Email address: clj@fjnu.edu.cn. }}
\author[b]{Changfeng Ma\thanks{Corresponding author. Email address:  macf@fjnu.edu.cn.}}
\affil[a]{School of Mathematics and Statistics \& Key Laboratory of Analytical Mathematics and Applications (Ministry of Education) \& Fujian Provincial Key Laboratory of Statistics and Artificial Intelligence, Fujian Normal University, Fuzhou, 350117, P.R. China}
\affil[b]{School of Artificial Intelligence, Xiamen Institute of Technology, Xiamen, 361021, P.R.  China}

\begin{document}
\date{\today}
\maketitle

\begin{abstract}
A parameter-free method, namely the generalization of the Gauss-Seidel (GGS) method, is developed to solve generalized absolute value equations. Convergence of the proposed method is analyzed. Numerical results are given to demonstrate the effectiveness and efficiency of the GGS method. Some results in the recent work of Edalatpour et al. \cite{edhs2017} are extended.

\end{abstract}

\section{Introduction}
Consider the system of generalized absolute value equations (GAVEs) of the form
\begin{equation}\label{eq:gave}
Ax - B\vert x \vert = b,
\end{equation}
where~$A,B \in \mathbb{R}^{n\times n}$ and $b\in\mathbb{R}^n$ are given, and $x\in\mathbb{R}^n$ is unknown. Here, the notation $|x|$ stands for the componentwise absolute value of $x$. To the best of our knowledge,  GAVEs~\eqref{eq:gave} were first formally studied by by Rohn in 2004 \cite{rohn2004}. GAVEs~\eqref{eq:gave} has become increasingly important in research because they are equivalent to the classic linear complementarity problem (LCP) \cite{cops2009,mame2006}. As demonstrated in \cite{mang2007}, solving the general GAVEs~\eqref{eq:gave} is NP-hard. Furthermore, when a solution exists, checking if it has a unique solution or more than one is an NP-complete problem \cite{pork2009}. Notwithstanding, conditions under which GAVEs~\eqref{eq:gave} has a unique solution for any $b\in \mathbb{R}^n$ are constructed in \cite{rohn2009,wuli2020,rohf2014,wush2021,pork2009,mezz2020,love2013,acha2018,kdhm2024}.

When GAVEs~\eqref{eq:gave} has a solution, researchers have worked on developing efficient algorithms to approximate these solutions.
Clearly, if matrix $B$ is nonsingular, we can transform GAVEs \eqref{eq:gave} into the system of absolute value equations (AVEs)
\begin{equation}\label{eq:ave}
Ax-|x|=b.
\end{equation}
In this sense, though it may be lack of efficiency due to the inversion of $B$, almost all numerical algorithms for solving AVEs~\eqref{eq:ave} can be theoretically adjusted to solve GAVEs~\eqref{eq:gave}.
However, when $B$ is singular, existing numerical methods for AVEs~\eqref{eq:ave} may not work for GAVEs~\eqref{eq:gave}. This motivates us to develop new efficient methods for solving the general GAVEs~\eqref{eq:gave}, in which $B$ is possible singular.

Over the past two decades, numerous numerical algorithms have been establishing to solve GAVEs~\eqref{eq:gave}. One approach is the generalized Newton (GN) method \cite{huhz2011}, which inherits the spirit of Mangasarian in \cite{mang2009}.
Subsequently, weaker convergent results of GN for GAVEs~\eqref{eq:gave} are investigated in \cite{lilw2018}. A problem with the GN method is that the Jacobian matrix may change during iterations, causing computational difficulties, especially for large-scale or ill-conditioned problems.
To remedy this issue, Wang et al. \cite{wacc2019} proposed a modified Newton-type (MN) iteration method by separating the differentiable and non-differentiable parts of GAVEs \eqref{eq:gave}. This method is described in Algorithm~\ref{alg:mn}.

\begin{algorithm}
\caption{The MN iteration method \cite{wacc2019}}\label{alg:mn}
Let $\Omega\in\mathbb{R}^{n\times n}$ be a  semi-definite matrix such that $A+\Omega$ is nonsingular. Given an initial vector $x^{(0)}\in\mathbb{R}^n$, for $k=0,1,2,\ldots$ until the iteration sequence $\{x^{(k)}\}^\infty_{k=0}$ is convergent, compute
\begin{equation*}
x^{(k+1)}=(A+\Omega)^{-1}(\Omega x^{(k)}+B|x^{(k)}|+b).
\end{equation*}
\end{algorithm}

Notably, when $\Omega=0$, the MN iteration reduces to the Picard iteration \cite{rohf2014}. Using matrix splitting techniques, researchers developed the Newton-based matrix splitting (NMS) method for GAVEs~\eqref{eq:gave} \cite{zhwl2021}, which extends the MN method.
Recently, the NMS iterative method is further extended in \cite{soso2023,zcss2024,lich2025}. For more numerical algorithms for solving GAVEs~\eqref{eq:gave}, one can refer to \cite{lild2023,tazh2019,jizh2013,xiqh2024,dawz2024,cyhm2025} and the references therein.

We should emphasis that, many of the above-mentioned numerical methods for solving GAVEs \eqref{eq:gave} rely on specific parameters (or parameter matrices). The performance of these algorithms are highly sensitive to the choice of the parameters, and the optimal parameters are often not easy to be determined.
The main goal of this paper is to develop a parameter-free iterative method for solving GAVEs~\eqref{eq:gave}. To this end, we recall the generalization of the Gauss-Seidel (GGS) iterative method for solving AVEs~\eqref{eq:ave} \cite{edhs2017}. We will extend the GGS iterative method to solve GAVEs~\eqref{eq:gave} and analyze its convergence. Numerical results are given to demonstrate the efficiency of our method.

The rest of this paper is organized as follows. Some necessary terminology and useful lemmas are given in \Cref{sec:pre}. In \Cref{sec:ggs}, the GGS iterative method for GAVEs~\eqref{eq:gave} is proposed and its convergence is analyzed. The numerical experiments and the conclusion of this paper are given in \Cref{sec:Num} and \Cref{sec:conclusion}, respectively.

\section{Preliminaries}\label{sec:pre}
In this section, we briefly introduce some notations, definitions and lemmas, which are necessary for the later development.

For notation, we follows the following convention:

\begin{itemize}
  \item $\mathbb{R}^{m\times n}$ is the set of $m\times n$ real matrices, $\mathbb{R}^m = \mathbb{R}^{m\times 1}$, and $\mathbb{R} = \mathbb{R}^1$;

  \item For $A=(a_{ij})\in \mathbb{R}^{m\times n}$,  the absolute value of the matrix $A$ is denoted by $|A|=(|a_{ij}|)$; The infinity norm of the matrix $A$ is defined as
$$
\|A\|_\infty=\max\limits_{1\leq i\leq m}\sum\limits^{n}_{j=1}|a_{ij}|;
$$

  \item For the given matrices $A=(a_{ij})\in\mathbb{R}^{m\times n}$ and $B=(b_{ij})\in\mathbb{R}^{m\times n}$, $A\geq B(A>B)$ if $a_{ij}\geq b_{ij}(a_{ij}>b_{ij})$ holds for all $1\leq i\leq m$ and $1\leq j\leq n$;

  \item For $U\in \mathbb{R}^{n\times n}$, $\rho(U)$ denotes the spectral radius of $U$.
\end{itemize}

For $A=(a_{ij})\in\mathbb{R}^{n\times n}$, the comparison matrix of $A$ is defined as $\langle A\rangle=(\langle a_{ij}\rangle)$, where
\begin{equation*}
\langle a_{ij}\rangle=
\begin{cases}
  ~~|a_{ij}|,  &\mbox{for } i=j, \\
  -|a_{ij}|,  &\mbox{for }  i\neq j.
\end{cases}
\end{equation*}
The matrix $A$ is called a $Z$-matrix if $a_{ij}\leq 0$ holds for all $i\neq j$,  an $M$-matrix if $A$ is a $Z$-matrix and $A^{-1}\geq 0$, and an $H$-matrix if $\langle A\rangle$ is an $M$-matrix. In addition, the matrix $A$ is called an $H_+$-matrix if $A$ is an $H$-matrix and $a_{ij}>0$ holds for all $i=j$. For a given matrix $A\in\mathbb{R}^{n\times n}$, $A=M-N$ is called an $M$-splitting of the matrix $A$ if $M$ is an $M$-matrix and $N\geq 0$ \cite{schn1984}.

In the following, we introduce some useful lemmas.

\begin{lem}[e.g., \cite{zhyi2013}]\label{lem:rho}
Let $A\in \mathbb{R}^{n\times n}$ be an $M$-matrix. If $A=M-N$ is an $M$-splitting of $A$, then $\rho(M^{-1}N)<1$.
\end{lem}

\begin{lem}[{\cite[Page 46]{varg1962}}]\label{lem:2.1}
Let $x$ and $y$ be two vectors in $\mathbb{R}^n$, then $\||x|-|y|\|_\infty \leq \|x-y\|_\infty$.
\end{lem}

\begin{lem}[e.g., {\cite[Theorem~1.10]{saad2003}}]\label{lem:2.0}
For~$U\in\mathbb{R}^{n\times n}$,~$\lim\limits_{k\rightarrow+\infty} U^k=0$ if and only if~$\rho(U)<1$.
\end{lem}

\begin{lem}[e.g., {\cite[Theorem~1.11]{saad2003}}]\label{lem:2.01}
For~$U\in\mathbb{R}^{n\times n}$, the series~$\sum\limits_{k=0}^\infty U^k$ converges if and only if~$\rho(U)<1$ and we have~$\sum\limits_{k=0}^\infty U^k=(I-U)^{-1}$ whenever it converges.
\end{lem}

\begin{lem}[{\cite[Theorem~1.33]{saad2003}}]\label{lem:2.2}
Let $A\in\mathbb{R}^{n\times n}$ be an $M$-matrix and $B\in\mathbb{R}^{n\times n}$ be a $Z$-matrix. If $A\leq B$, then $B$ is an $M$-matrix.
\end{lem}

\begin{lem}[{\cite[Lemma 3.4]{frma1989}}]\label{lem:2.3}
If $A\in\mathbb{R}^{n\times n}$ is an $H$-matrix, then $|A^{-1}|\leq \langle A\rangle ^{-1}$.
\end{lem}

\begin{lem}[{\cite[Theorem~1.29]{saad2003}}]\label{lem:2.4}
If $A\in \mathbb{R}^{n\times n}$ is a nonnegative matrix. Then $I-A$ is nonsingular with $(I-A)^{-1} \geq 0$ if and only if $\rho(A)<1$.
\end{lem}

\begin{lem}[{\cite[Lemma 5]{edhs2017}}]\label{lem:ggs4ave}
Let $a\in \mathbb{R}$ and $a>1$. Then equation $ax - |x| -c = 0$ has only one solution in $\mathbb{R}$ for any $c\in \mathbb{R}$. If $c\ge 0$, then the solution is given by $x = \frac{c}{a-1}$. Otherwise, $x = \frac{c}{a+1}$ is the solution.
\end{lem}

\section{The GGS method for GAVEs}\label{sec:ggs}
In this section, we will present a GGS iterative method to solve GAVEs~\eqref{eq:gave}. To this end, we first recall the GGS iterative method for solving AVEs~\eqref{eq:ave} \cite{edhs2017}.

Let $$A=D_A-L_A-U_A,$$
where $D_A,-L_A$ and $-U_A$ are the diagonal, the strictly lower triangular and the strictly upper triangular parts of $A$, respectively. Then AVEs~\eqref{eq:ave} can be rewritten as the fixed-point equation
\begin{equation*}
(D_A-L_A)x-|x|=U_A x+b,
\end{equation*}
from which we can define the GGS iteration for solving AVEs~\eqref{eq:ave} as
\begin{equation}\label{eq:ggs4ave}
(D_A-L_A)x^{(k+1)}-|x^{(k+1)}|=U_Ax^{(k)}+b,\quad k=0,1,2,\ldots
\end{equation}
with $x^{(0)}$ being a given initial guess. In each iteration of the GGS method for solving AVEs~\eqref{eq:ave} (recall \eqref{eq:ggs4ave}), a lower triangular absolute value system
\begin{equation}\label{eq:tave}
(D_A-L_A)x-|x|=c
\end{equation}
should be solved with $c$ being a given vector. Based on Lemma~\ref{lem:ggs4ave}, under the condition that the diagonal entries of $A$ exceed unit, the solution of the system \eqref{eq:tave} can be uniquely determined and then the GGS iteration \eqref{eq:ggs4ave} is well-defined.

In the following, we will extend the GGS iteration \eqref{eq:ggs4ave} to solve GAVEs~\eqref{eq:gave}. Since $B$ may be singular, we cannot firstly transform GAVEs~\eqref{eq:gave} into AVEs~\eqref{eq:ave} and then use \eqref{eq:ggs4ave}. Instead, we also split $B$ as
$$B = D_B - L_B - U_B,$$
where $D_B,-L_B$ and $-U_B$ are the diagonal, the strictly lower triangular and the strictly upper triangular parts of $B$, respectively. Then GAVEs~\eqref{eq:gave} can be written as the following fixed-point equation
\begin{equation*}
(D_A-L_A)x-(D_B-L_B)|x|=U_A x-U_B|x|+b,
\end{equation*}
from which we can define the following GGS iteration for solving GAVEs~\eqref{eq:gave}:
\begin{equation}\label{eq:ggs4gave}
(D_A-L_A)x^{(k+1)}-(D_B-L_B)|x^{(k+1)}|=U_A x^{(k)}-U_B|x^{(k)}|+b,\quad k=0,1,2,\ldots,
\end{equation}
where $x^{(0)}$ is the initial guess. In each iteration of the GGS method for solving GAVEs~\eqref{eq:ave} (see \eqref{eq:ggs4gave}), a lower triangular generalized absolute value system
\begin{equation}\label{eq:tgave}
(D_A-L_A)x-(D_B-L_B)|x|=c
\end{equation}
should be solved with $c$ being a given vector. The following lemma gives conditions under which each equation of the system \eqref{eq:tgave} can be uniquely solvable for any $c\in \mathbb{R}^n$ and then the GGS iteration \eqref{eq:ggs4gave} is well-defined.

\begin{lem} \label{lem:3.1}
Let $a,b\in\mathbb{R}$ and $a>|b|$. Then the equation $ax-b|x|=c$ has a unique solution in $\mathbb{R}$ for any $c\in \mathbb{R}$. Specifically, if $c\geq 0$, the solution is $x=\frac{c}{a-b}$. Otherwise, the solution is $x=\frac{c}{a+b}$.
\end{lem}
\begin{proof}
The unique solvability follows from $a>|b|$ and \cite[Theorem 2]{rohn2009}.  If $c\geq 0$, it can be checked that the unique solution is $x=\frac{c}{a-b}$. Otherwise, the unique solution is $x=\frac{c}{a+b}$.
\end{proof}

\begin{rem}
Obviously, if $b=1$, then Lemma~\ref{lem:3.1} reduces to Lemma~\ref{lem:ggs4ave}.
\end{rem}

Based on Lemma~\ref{lem:3.1}, the GGS iteration method for solving GAVEs~\eqref{eq:gave} with $ D_A>|D_B|$ can be summarized as in \Cref{alg:1}.

\begin{algorithm}[h]
\caption{The GGS method for GAVEs}\label{alg:1}
\begin{algorithmic}[1]
    \REQUIRE matrices $A,B\in\mathbb{R}^{n\times n}$ with $ D_A>|D_B|$  and $b\in\mathbb{R}^n$. Take the initial vector $x^{(0)}\in\mathbb{R}^n$. Set $k:=0$.

    \ENSURE an approximate solution of GAVEs \eqref{eq:gave}.

    \REPEAT

    \STATE For $i=1$

    \STATE $s=b_1-\sum\limits_{j=2}^{n} a_{ij}x_j^{(k)}
 +\sum\limits_{j=2}^{n} b_{ij}|x_j^{(k)}|$

 \STATE If $s\geq 0$, then $x_1^{(k+1)}=s/(a_{11}-b_{11})$

  \STATE Else $x_1^{(k+1)}=s/(a_{11}+b_{11})$

    \STATE For $i=2,\ldots,n-1$, Do

    \STATE $s=b_{i}-\sum\limits_{j=1}^{i-1} a_{ij}x_j^{(k+1)}
    +\sum\limits_{j=1}^{i-1} b_{ij}|x_j^{(k+1)}|
 -\sum\limits_{j=i+1}^{n} a_{ij}x_j^{(k)}
 +\sum\limits_{j=i+1}^{n} b_{ij}|x_j^{(k)}|$

 \STATE If $s\geq 0$, then $x_i^{(k+1)}=s/(a_{ii}-b_{ii})$

  \STATE Else $x_i^{(k+1)}=s/(a_{ii}+b_{ii})$

 \STATE EndDo

 \STATE For $i=n$

 \STATE $s=b_n-\sum\limits_{j=1}^{n-1} a_{ij}x_j^{(k+1)}
    +\sum\limits_{j=1}^{n-1} b_{ij}|x_j^{(k+1)}|$

 \STATE If $s\geq 0$, then $x_n^{(k+1)}=s/(a_{nn}-b_{nn})$

  \STATE Else $x_n^{(k+1)}=s/(a_{nn}+b_{nn})$

    \UNTIL{convergence};
    \RETURN the last $x^{(k+1)}$ as the approximation to the solution of GAVEs~\eqref{eq:gave}.
\end{algorithmic}
\end{algorithm}

%
%
%
%
%
%
%
%

In the following, we will discuss the convergence properties of the GGS
method.

\begin{thm}\label{thm:3.1}
Suppose that $D_A>|D_B|$ and the matrix $D_A-|L_A|-D_B-|L_B|$ is strictly row diagonally dominant. If
\begin{equation}\label{eq:con}
\|(D_A-L_A)^{-1}U_A\|_\infty+\|(D_A-L_A)^{-1}U_B\|_\infty < 1-\|(D_A-L_A)^{-1}(D_B+|L_B|)\|_\infty,
\end{equation}
then the iteration sequence $\{x^{(k)}\}^\infty_{k=0}$ generated by \eqref{eq:ggs4gave} converges to the unique solution $x_*$ of the GAVEs \eqref{eq:gave} for any initial vector $x^{(0)}\in\mathbb{R}^n$.
\end{thm}

\begin{proof}

Firstly, we prove that $\|(D_A-L_A)^{-1}(D_B+|L_B|)\|_\infty<1$. Obviously, if $L_A=L_B=0$, then $\|(D_A-L_A)^{-1}(D_B+|L_B|)\|_\infty=\|D_A^{-1}D_B\|_\infty<1$ since $D_A>|D_B|$. We assume that $L_A$ and $L_B$ are not simultaneously zero in the following.

Since $D_A>|D_B|$ and $D_A-|L_A|-D_B-|L_B|$ is strictly row diagonally dominant,
we obtain
\begin{equation*}
|(D_A-|L_A|-D_B-|L_B|)_{ii}|>\sum\limits_{j\neq i}|(D_A-|L_A|-D_B-|L_B|)_{ij}|,~\forall i.
\end{equation*}
Then $0\leq |L_A|e<(D_A-D_B-|L_B|)e$, or equivalently,
\begin{equation}\label{eq:1}
D_A^{-1} (D_B+|L_B|)e<(I-|F|)e,
\end{equation}
where $F=D_A^{-1}L_A$ and $e=(1,1,\ldots,1)^\top$. In addition, since $F$ is a strictly lower triangle matrix, we have
\begin{align}\label{eq:2}
0 &\leq |(I-F)^{-1}|=|I+F+F^2+\ldots+F^{n-1}|\nonumber\\
&\leq I+|F|+|F|^2+\ldots+|F|^{n-1}=(I-|F|)^{-1}.
\end{align}
Hence, from \eqref{eq:1} and \eqref{eq:2}, we obtain
\begin{align*}
 |(D_A-L_A)^{-1}(D_B+|L_B|)|e &= |(D_A-D_A F)^{-1}(D_B+|L_B|)|e\\ &=|(I-F)^{-1}D_A^{-1}(D_B+|L_B|)|e\\
&\leq |(I-F)^{-1}||D_A^{-1}(D_B+|L_B|)|e<(I-|F|)^{-1}(I-|F|)e\\
&=e.
\end{align*}
This implies
\begin{equation*}
\|(D_A-L_A)^{-1}(D_B+|L_B|)\|_\infty<1.
\end{equation*}

From \eqref{eq:ggs4gave}, we have
\begin{align*}
(D_A-L_A)(x^{(k+1)}-x^{(k)})&=(D_B+|L_B|) (|x^{(k+1)}|-|x^{(k)}|)\\
&+U_A (x^{(k)}-x^{(k-1)})-U_B (|x^{(k)}|-|x^{(k-1)}|),
\end{align*}
from which and Lemma~\ref{lem:2.1}, we get
\begin{align*}
\|x^{(k+1)}-x^{(k)}\|_\infty &\leq \|(D_A-L_A)^{-1}(D_B+|L_B|)\|_\infty \|x^{(k+1)}-x^{(k)}\|_\infty\\
&\qquad +\|(D_A-L_A)^{-1}U_A\|_\infty \|x^{(k)}-x^{(k-1)}\|_\infty\\
&\qquad +\|(D_A-L_A)^{-1}U_B\|_\infty \|x^{(k)}-x^{(k-1)}\|_\infty,
\end{align*}
which  is equivalent to
\begin{align*}
&(1-\|(D_A-L_A)^{-1}(D_B+|L_B|)\|_\infty)\|x^{(k+1)}-x^{(k)}\|_\infty \\
&\leq (\|(D_A-L_A)^{-1}U_A\|_\infty+\|(D_A-L_A)^{-1}U_B\|_\infty)\|x^{(k)}-x^{(k-1)}\|_\infty.
\end{align*}

Since $\|(D_A-L_A)^{-1}(D_B+|L_B|)\|_\infty<1$, then
\begin{equation*}
\|x^{(k+1)}-x^{(k)}\|_\infty \leq \frac{\|(D_A-L_A)^{-1}U_A\|_\infty+\|(D_A-L_A)^{-1}U_B\|_\infty}{1-\|(D_A-L_A)^{-1}(D_B+|L_B|)\|_\infty}
\|x^{(k)}-x^{(k-1)}\|_\infty.
\end{equation*}

We denote $t=\frac{\|(D_A-L_A)^{-1}U_A\|_\infty+\|(D_A-L_A)^{-1}U_B\|_\infty}{1-\|(D_A-L_A)^{-1}(D_B+|L_B|)\|_\infty}$, then $t<1$ if \eqref{eq:con} holds. For each $m\geq 1$, with $t<1$, it follows from Lemma \ref{lem:2.0} and Lemma \ref{lem:2.01} that
\begin{align*}
 \|x^{(k+m)}-x^{(k)}\|_\infty&=\left\|\sum_{j=0}^{m-1}(x^{(k+j+1)}-x^{(k+j)})\right\|_\infty
\leq\sum_{j=0}^{m-1} \|x^{(k+j+1)}-x^{(k+j)}\|_\infty\\
&\leq \sum_{j=0}^{\infty}t^{j+1}\|x^{(k)}-x^{(k-1)}\|_\infty
=\frac{t}{1-t}\|x^{(k)}-x^{(k-1)}\|_\infty\\
&\leq \frac{t^k}{1-t} \|x^{(1)}-x^{(0)}\|_\infty\rightarrow 0
~~(\text{as}\quad k\rightarrow \infty).
\end{align*}
Therefore, $\{x^{(k)}\}_{k=0}^{\infty}$ is a Cauchy sequence and convergent in $\mathbb{R}^n$. Let $\lim\limits_{k\rightarrow\infty} x^{(k)} =x_{*}$, it follows from \eqref{eq:ggs4gave} that
\begin{equation*}
(D_A-L_A)x_*-(D_B-L_B)|x_*|=U_A x_*-U_B|x_*|+b,
\end{equation*}
that is, $Ax_*-B|x_*|=b$. Hence,  $x_*$ is a solution to GAVEs \eqref{eq:gave}.

To prove uniqueness of the solution, let  $y_*$ be another solution of GAVEs \eqref{eq:gave}. Then we have
\begin{align*}
\|x_*-y_*\|_\infty &\leq \|(D_A-L_A)^{-1}(D_B+|L_B|)\|_\infty \|x_*-y_*\|_\infty\\
&\qquad +\|(D_A-L_A)^{-1}U_A\|_\infty \|x_*-y_*\|_\infty
+\|(D_A-L_A)^{-1}U_B\|_\infty \|x_*-y_*\|_\infty\\
&<\|(D_A-L_A)^{-1}(D_B+|L_B|)\|_\infty \|x_*-y_*\|_\infty\\
&\qquad +(1-\|(D_A-L_A)^{-1}(D_B+|L_B|)\|_\infty)\|x_*-y_*\|_\infty\\
&=\|x_*-y_*\|_\infty,
\end{align*}
which is a contradiction whenever $x_*\neq y_*$.
\end{proof}

\begin{rem}
When $B=I$, Theorem \ref{thm:3.1} can recover the result of \cite[Theorem 3]{edhs2017} without using the assumption of the solvability of AVEs~\eqref{eq:ave}.
\end{rem}

\begin{thm}\label{thm:3.2}
If $D_A>|D_B|$ and the matrix $\langle A \rangle-|B|$ is an $M$-matrix, then the iteration sequence $\{x^{(k)}\}^\infty_{k=0}$ generated by \eqref{eq:ggs4gave} converges to the unique solution $x_*$ of the GAVEs~\eqref{eq:gave} for any initial vector $x^{(0)}\in\mathbb{R}^n$.
\end{thm}

\begin{proof}
Since $\langle A\rangle-|B|=\langle D_A-L_A\rangle-|U_A|-|B|$ is an $M$-matrix, then $\langle D_A-L_A\rangle$ is an $M$-matrix by Lemma~\ref{lem:2.2}.
It follows from Lemma \ref{lem:2.3} that
\begin{equation*}
|(D_A-L_A)^{-1}| \leq \langle D_A-L_A\rangle^{-1}.
\end{equation*}

From \eqref{eq:ggs4gave}, we have
\begin{align*}
 &|x^{(k+1)}-x^{(k)}|\\\leq &|(D_A-L_A)^{-1}|(|D_B-L_B||x^{(k+1)}-x^{(k)}|
 +|U_A||x^{(k)}-x^{(k-1)}|+|U_B||x^{(k)}-x^{(k-1)}|)\\
 \leq&\langle D_A-L_A\rangle^{-1}|D_B-L_B||x^{(k+1)}-x^{(k)}|
+\langle D_A-L_A\rangle^{-1}(|U_A|+|U_B|)|x^{(k)}-x^{(k-1)}|,
\end{align*}
or equivalently,
\begin{equation*}
(I-\langle D_A-L_A\rangle^{-1}|D_B-L_B|)|x^{(k+1)}-x^{(k)}|
\leq \langle D_A-L_A\rangle^{-1}(|U_A|+|U_B|)|x^{(k)}-x^{(k-1)}|.
\end{equation*}

Since $\langle A\rangle-|B|=\langle D_A-L_A\rangle-|D_B-L_B|-|U_A|-|U_B|$ is an $M$-matrix, $\langle D_A-L_A\rangle-|D_B-L_B|$ is also an $M$-matrix by Lemma \ref{lem:2.2}. Thus $\langle D_A-L_A\rangle-|D_B-L_B|$ is an $M$-splitting and it follows from Lemma \ref{lem:rho} that  $\rho(\langle D_A-L_A\rangle^{-1}|D_B-L_B|)<1$.

From Lemma \Cref{lem:2.4}, we have $(I-\langle D_A-L_A\rangle^{-1}|D_B-L_B|)^{-1}\geq 0$. Then
\begin{align*}
|x^{(k+1)}-x^{(k)}|&\leq (I-\langle D_A-L_A\rangle^{-1}|D_B-L_B|)^{-1}\langle D_A-L_A\rangle^{-1}(|U_A|+|U_B|)|x^{(k)}-x^{(k-1)}|\\
&=(\langle D_A-L_A\rangle-|D_B-L_B|)^{-1}(|U_A|+|U_B|)|x^{(k)}-x^{(k-1)}|.
\end{align*}
Next we verify that $\rho(\hat{L})<1$, where $\hat{L}=\hat{M}^{-1}\hat{N}$ with
\begin{equation*}
\hat{M}=\langle D_A-L_A\rangle-|D_B-L_B|,~\hat{N}=|U_A|+|U_B|,
\end{equation*}
and
$$
\hat{M}-\hat{N}=\langle D_A-L_A\rangle-|D_B-L_B|-|U_A|-|U_B|=\langle A\rangle-|B|.
$$
From the assumption that $\langle A\rangle-|B|$ is an $M$-matrix and $\hat{M}-\hat{N}$ is an $M$-splitting, it follows from Lemma~\ref{lem:rho} that $\rho(\hat{L})<1$.

For each $m\geq1$, with $\rho(\hat{L})<1$, it follows Lemma \ref{lem:2.0} and Lemma \ref{lem:2.01} that
\begin{align*}
|x^{(k+m)}-x^{(k)}|&=\left|\sum_{j=0}^{m-1}(x^{(k+j+1)}-x^{(k+j)})\right|
\leq\sum_{j=0}^{m-1} |x^{(k+j+1)}-x^{(k+j)}|\\
&\leq \sum_{j=0}^{\infty}\hat{L}^{j+1}|x^{(k)}-x^{(k-1)}|
=(I-\hat{L})^{-1}\hat{L}|x^{(k)}-x^{(k-1)}|\\
&\leq (I-\hat{L})^{-1}\hat{L}^k |x^{(1)}-x^{(0)}|\rightarrow 0
~~(\text{as}\quad k\rightarrow \infty).
\end{align*}
Therefore, $\{x^{(k)}\}_{k=0}^{\infty}$ is a Cauchy sequence and convergent in $\mathbb{R}^n$. Let $\lim\limits_{k\rightarrow\infty} x^{(k)} =x_{*}$, it follows from \eqref{eq:ggs4gave} that
\begin{equation*}
(D_A-L_A)x_*-(D_B-L_B)|x_*|=U_A x_*-U_B|x_*|+b,
\end{equation*}
which implies that $x_*$ is a solution to GAVEs \eqref{eq:gave}.

Next we prove the uniqueness of the solution. Let $\hat{x}$ be any solution of
GAVEs~\eqref{eq:gave} and take $x^{(0)}=\hat{x}$, then we have
\begin{equation}\label{eq:x0}
(D_A-L_A)x^{(0)}-(D_B-L_B)|x^{(0)}|=U_A x^{(0)}-U_B|x^{(0)}|+b.
\end{equation}
From \eqref{eq:ggs4gave}, we have
\begin{equation*}
(D_A-L_A)x^{(1)}-(D_B-L_B)|x^{(1)}|=U_A x^{(0)} -U_B|x^{(0)}|+b,
\end{equation*}
which has the same right hand side as that of \eqref{eq:x0}. According to Lemma~\ref{lem:3.1}, we have $x^{(1)} = x^{(0)}=\hat{x}$. Proceeding in the same way, we obtain $x^{(k)}=\hat{x}$ for all $k\geq 0$. The convergence of $x^{(k)}$ to $x_*$ implies that $\hat{x}=x_*$. Thus $x_*$ is the unique solution.
\end{proof}

\begin{rem}
When $B=I$, Theorem \ref{thm:3.2} can recover the result of \cite[Theorem 4]{edhs2017} without using the assumption of the solvability of AVEs~\eqref{eq:ave}.
\end{rem}

\section{Numerical results}\label{sec:Num}

In this section, we present numerical experiments to evaluate the performance of the proposed GGS method for solving GAVEs \eqref{eq:gave} and compare it with some existing methods. All experiments are conducted by using MATLAB (version 9.10 (R2021a)) on a personal computer with IntelCore (TM) i7 CPU 2.60 GHz, 16.0 GB memory.


\subsection{GAVEs arising from LCP}\label{sec:lcp}
In this subsection, the GGS method is used to solve GAVEs~\eqref{eq:gave} arising from LCP. LCP is to find a vector $z\in \mathbb{R}^n$ such that
\begin{equation}\label{eq:lcp}
z\geq 0,\quad w:=Mz+q\geq 0,\quad z^\top w=0,
\end{equation}
where $M\in\mathbb{R}^{n\times n}$ and $q\in\mathbb{R}^n$ are known. One notable method for solving LCP \eqref{eq:lcp} is the accelerated modulus-based matrix splitting iteration method \cite{zhyi2013}. This method reformulates LCP \eqref{eq:lcp} as an implicit fixed-point equation, which can be equivalently expressed as GAVEs:
\begin{equation}\label{eq:gave1}
(M+\Omega)x-(\Omega-M)|x|=-\gamma q,
\end{equation}
where $\Omega$ is an $n\times n$ positive diagonal matrix and $\gamma>0$ is a given parameter. By splitting $M=M_1-N_1=M_2-N_2$, the accelerated modulus-based matrix splitting iteration
\begin{equation*}
(M_1+\Omega)x^{(k+1)}=N_1 x^{(k)}+(\Omega-M_2)|x^{(k)}|+N_2|x^{(k+1)}|-\gamma q
\end{equation*}
is developed \cite{zhyi2013}. Specially, if $M_1=D_M-L_M,N_1=U_M,M_2=D_M-U_M,N_2=L_M$, it gives the accelerated modulus-based Gauss-Seidel (AMGS) iteration method, see Algorithm~\ref{alg:amgs}. Here, $D_M,-L_M$ and $-U_M$ are the diagonal, the strictly lower triangular and the strictly upper triangular parts of $M$, respectively.




\begin{algorithm}
\caption{The AMGS iteration method \cite{zhyi2013}}\label{alg:amgs} Given an initial vector $x^{(0)}\in\mathbb{R}^{n}$, for $k=0,1,2,\ldots$ until the iteration sequence $\{z^{(k)}\}^{+\infty}_{k=0}$ is convergent, compute $x^{(k+1)}\in\mathbb{R}^n$ by solving the linear system
\begin{equation}\label{eq:amgs}
(D_M-L_M+\Omega)x^{(k+1)}=U_Mx^{(k)}+(\Omega-D_M+U_M)|x^{(k)}|+L_M|x^{(k+1)}|-\gamma q,
\end{equation}
and set
$$
z^{(k+1)}=\frac{1}{\gamma}(|x^{(k+1)}|+x^{(k+1)}).
$$
\end{algorithm}



Now we use our GGS method to solve GAVEs \eqref{eq:gave1}, whose iteration is
\begin{equation}\label{eq:ggs1}
(D_M-L_M+\Omega)x^{(k+1)}-(\Omega-D_M+L_M)|x^{(k+1)}|=U_M x^{(k)}+U_M|x^{(k)}|-\gamma q.
\end{equation}
Unlike \eqref{eq:amgs}, \eqref{eq:ggs1} inherits the structure of GAVEs.

In the following, we will compare the performances of GGS and AMGS when they are used to solve LCP. In the numerical results, we report the number of iteration steps (denoted by `IT'), elapsed CPU time in seconds (denoted by `CPU'), and
$$
\text{RES}(z^{(k)}):=\|\min (Mz^{(k)}+q,~z^{(k)})\|_2,
$$
where $z^{(k)}$ is the $k$th approximate solution to the LCP and the minimum is taken componentwise. The following tables record the average results of IT, CPU, and RES after ten experiments. The same goes to Example~\ref{exam:5.2} and Example~\ref{exam:5.3} (except RES).

\begin{exam}[\cite{zhyi2013}]\label{exam:5.1}
Let $m$ be a positive integer and $n=m^2$. Consider LCP \eqref{eq:lcp}, in which $M\in\mathbb{R}^{n\times n}$ is given by $M=\hat{M}+\mu I$ and $q\in\mathbb{R}^n$ is given by $q=-Mz^*$, where
\begin{equation*}
\hat{M}=tridiag(-I,S,-I)=
\begin{pmatrix}
  S & -I & 0 & \cdots & 0 & 0 \\
  -I & S & -I & \cdots & 0 & 0 \\
  0 & -I & S & \cdots & 0 & 0 \\
  \vdots & \vdots & \vdots & \ddots & \vdots & \vdots \\
  0 & 0 & 0 & \cdots & S & -I \\
  0 & 0 & 0 & \cdots & -I & S
\end{pmatrix}\in\mathbb{R}^{n\times n}
\end{equation*}
is a block-tridiagonal matrix,
\begin{equation*}
S=tridiag(-1,4,-1)=
\begin{pmatrix}
 4 & -1 & 0 & \cdots & 0 & 0 \\
  -1 & 4 & -1 & \cdots & 0 & 0 \\
  0 & -1 & 4 & \cdots & 0 & 0 \\
  \vdots & \vdots & \vdots & \ddots & \vdots & \vdots \\
  0 & 0 & 0 & \cdots & 4 & -1 \\
  0 & 0 & 0 & \cdots & -1 & 4
\end{pmatrix}\in\mathbb{R}^{m\times m}
\end{equation*}
is a tridiagonal matrix, and
$z^*=(1,2,1,2,\ldots)^\top\in\mathbb{R}^n$ is
the unique solution of  LCP~\eqref{eq:lcp}.

For this example, we set $x^{(0)}=(1,0,1,0,\ldots)^\top\in\mathbb{R}^n$, $\gamma=1$ and $\Omega=\theta D_M$. Both methods are stopped if $\text{RES}(z^{(k)})\leq 10^{-5}$ or $\text{IT}\geq 100$. For the AMGS method, as noted in \cite{zhyi2013}, the iteration parameter $\theta$ used in the actual computations is obtained experimentally by minimizing the corresponding iteration steps.
We tested values within the range $[0:0.01:2]$ and selected the value that resulted in the fewest iterations. It is important to note that the total computational cost of AMGS consists of two components: the time required to find the optimal parameter (CPU$_{opt}$) and the time for the iteration with the optimal parameter (CPU). Numerical results are shown in Table~\ref{table1}, from which we can see that AMGS with the optimal parameter $\theta_{opt}$ needs only $2$-$3$ fewer iterations than that of GGS. However, as the problem dimension $m$ increases, the computational time required to determine the optimal parameter of AMGS also grows significantly. In addition, Figure~\ref{fig1} shows that the performance of GGS is independent of $\theta$. Indeed, for this example, we observe that $s\geq0$ throughout the iteration process, and the update formula becomes
\begin{equation*}
x_i^{(k+1)}=\frac{s}{a_{ii}-b_{ii}}=\frac{s}{(M + \Omega)_{ii}-(\Omega-M)_{ii}}=\frac{s}{2M_{ii}},
\end{equation*}
which is independent of the parameter $\theta$.


\begin{table}[htbp]
\centering
\caption{Numerical comparison of AMGS and GGS with $\mu=4$ in Example~\ref{exam:5.1}.}\label{table1}
\begin{tabular}{ccccccccc}\hline
&Method     &$m$                     \\\cline{3-7}
&{}            &60         &70        &80         &90      &100     \\\hline
&AMGS \\
&$\theta_{opt}$  &0.80   &0.80    &0.80     &0.81     &0.81  \\
&IT                     &13        &13      &13    &13     &13    \\
&CPU$_{opt}$   &1.9862    &3.2808  &4.6706  &6.4324  &9.6480  \\
&CPU                 &0.6006   &0.8467  &1.1913  &1.6087  &2.1141  \\
&RES  &8.0448e-06  &8.7655e-06  &9.4324e-06  &7.8475e-06  &8.3151e-06 \\
&{}\\
&GGS \\
&IT       &15         &15          &16       &16    &16   \\
&CPU    &0.6761   &0.9637  &1.3992  &2.0001  &2.5369  \\
&RES  &7.4541e-06  &8.8482e-06   &3.3955e-06   &3.8602e-06  &4.3248e-06      \\\hline
\end{tabular}
\end{table}

\begin{figure}[htbp]
\centering
\begin{minipage}{0.48\linewidth}
\centering
\includegraphics[scale=0.55]{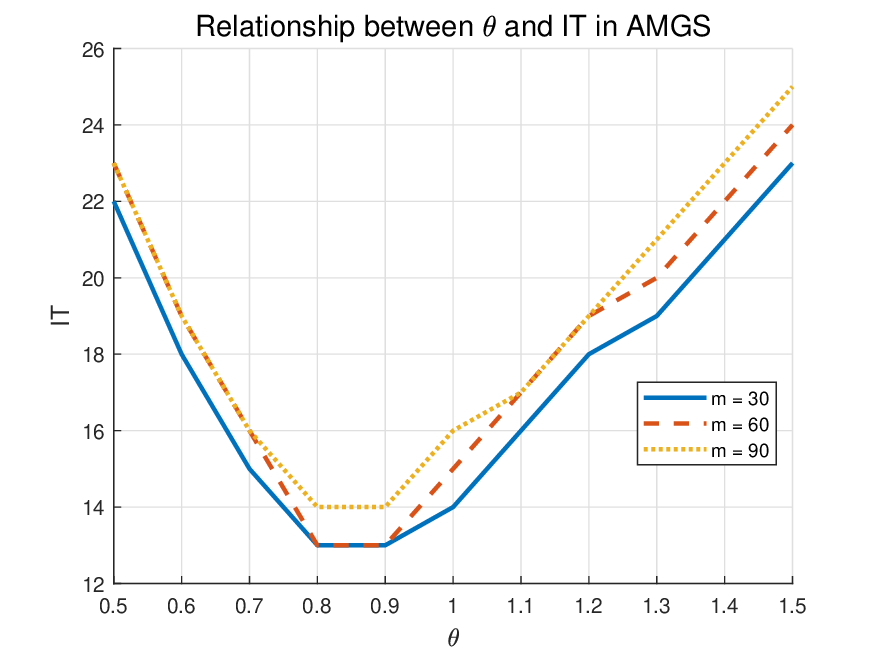}
\end{minipage}
\begin{minipage}{0.48\linewidth}
\centering
\includegraphics[scale=0.55]{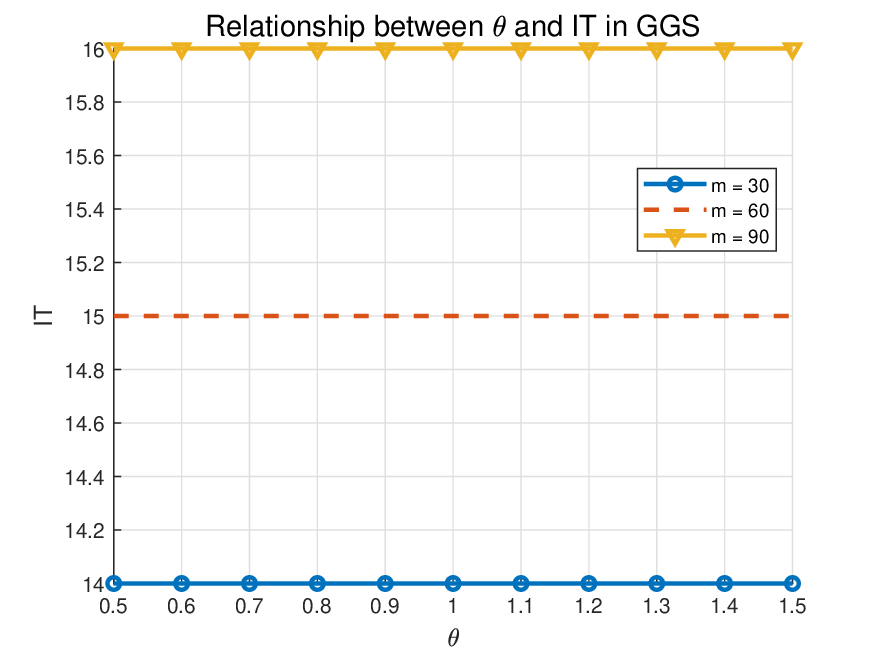}
\end{minipage}
\caption{Parametric influence on GGS and AMGS in Example~\ref{exam:5.1}}\label{fig1}
\end{figure}
\end{exam}

To further verify the parameter insensitivity of the GGS method, we consider another example in which the value of $s$ frequently becomes negative during the iteration process.

\begin{exam}\label{exam:5.2}
Let $m$ be a positive integer and $n=m^2$. Consider LCP \eqref{eq:lcp}, in which $M\in\mathbb{R}^{n\times n}$ is given by $M=\hat{M}+\mu I$ and $q\in\mathbb{R}^n$ is given by $q=-Mz^*$, where
$
\hat{M}=tridiag(-I,S,-I)
$
is a block-tridiagonal matrix, $S=tridiag(-1,4,-1)$ is a tridiagonal matrix, and
$z^*=(1,10,1,10,\ldots)^\top\in\mathbb{R}^n$ is
the unique solution of LCP \eqref{eq:lcp}.

Similarly, as shown in Table~\ref{table2}, the AMGS method with the optimal parameter $\theta_{opt}$ requires approximately $3$-$4$ fewer iterations than the GGS method. However, the computational overhead associated with determining the optimal parameter increases significantly as the problem dimension $m$ becomes larger. In this example, although many cases with $s < 0$ are observed, Figure \ref{fig2} demonstrates that the performance of the GGS method remains unaffected by the parameter $\theta$. However, the theory guarantee for this phenomenon needs further study.

\begin{table}[htbp]
\centering
\caption{Numerical comparison of AMGS and GGS with $\mu=4$ in Example~\ref{exam:5.2}.}\label{table2}
\begin{tabular}{ccccccccc}\hline
&Method     &$m$                     \\\cline{3-7}
&{}                     &60         &70        &80         &90      &100     \\\hline
&AMGS \\
&$\theta_{opt}$  &0.79   &0.80    &0.80     &0.80     &0.80  \\
&IT                     &14        &14      &14    &14     &14    \\
&CPU$_{opt}$   &1.6501    &2.4690  &3.3423  &4.5097  &6.3419  \\
&CPU                 &0.6357   &0.9342  &1.2707  &1.7339  &2.4778 \\
&RES  &9.4830e-06  &7.9177e-06  &8.5292e-06  &9.1031e-06  &9.6462e-06 \\
&{}\\
&GGS \\
&IT       &17         &17          &17       &17         &18   \\
&CPU    &0.7925   &1.1012  &1.5353  &2.0940  &3.2469  \\
&RES  &5.9028e-06  &7.0313e-06   &8.1598e-06   &9.2882e-06  &3.4576e-06      \\\hline
\end{tabular}
\end{table}

\begin{figure}[htbp]
\centering
\begin{minipage}{0.48\linewidth}
\centering
\includegraphics[scale=0.55]{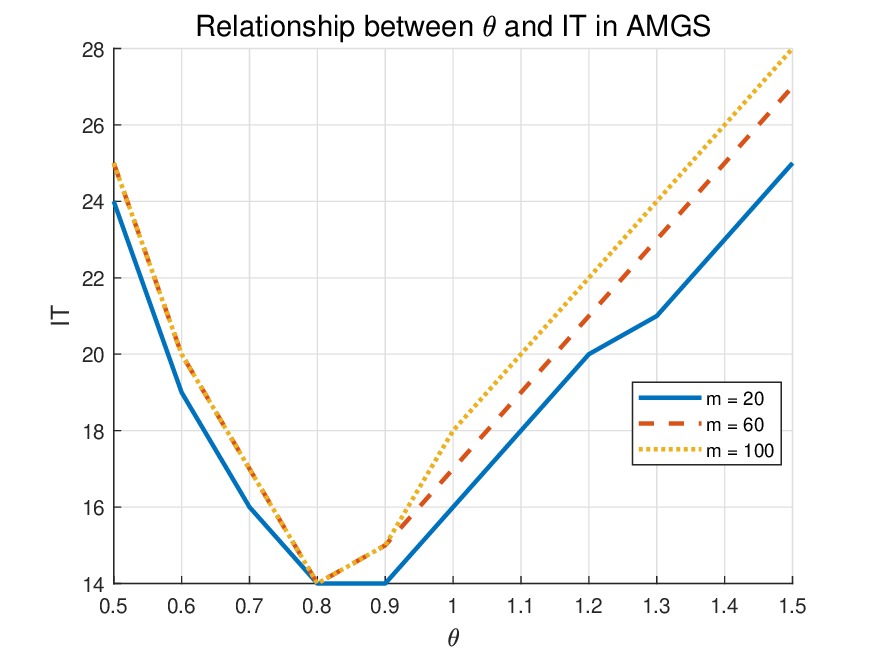}
\end{minipage}
\begin{minipage}{0.48\linewidth}
\centering
\includegraphics[scale=0.55]{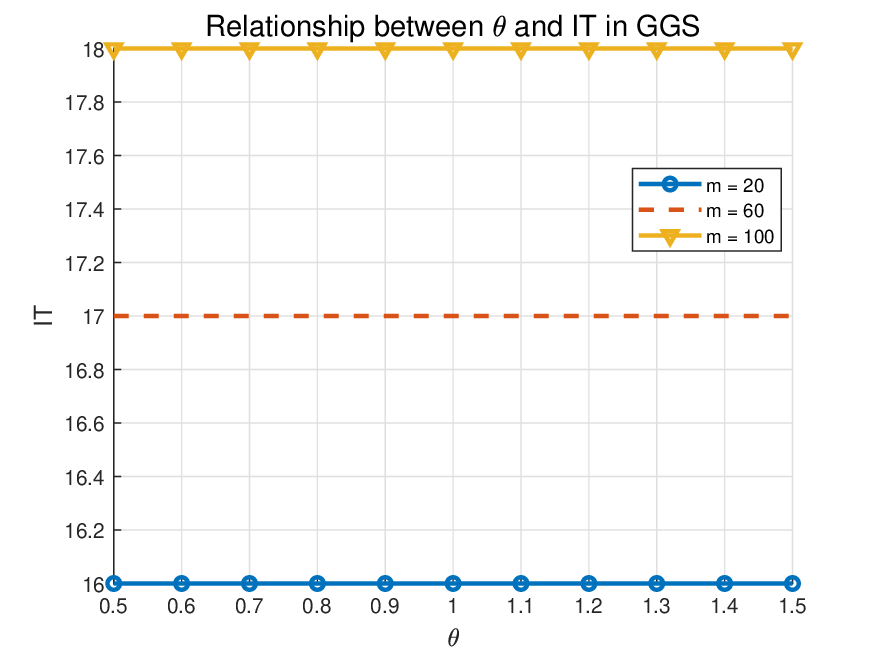}
\end{minipage}
\caption{Parametric influence on GGS and AMGS in Example~\ref{exam:5.2}}\label{fig2}
\end{figure}
\end{exam}

\subsection{GAVEs with $B$ being singular}
In this subsection, we evaluate the performance of our proposed GGS method for solving GAVEs~\eqref{eq:gave} with $B$ being singular. We will compare our GGS method with the following nine methods:

\noindent\textbf{GNMS}: the generalized Newton-based matrix splitting method \cite{lich2025}
\begin{equation*}
\begin{cases}
y^{(k+1)}=(1-\tau)y^{(k)}+\tau Q_1^{-1}(Q_2y^{(k)}+|x^{(k)}|),\\
x^{(k+1)}=M^{-1}(Nx^{(k)}+BQ_1y^{(k+1)}-BQ_2y^{(k)}+b),
\end{cases}
\end{equation*}
where $Q_1=10I,Q_2=0.5I,M=D_A-\frac{3}{4}L_A$ and $N=\frac{1}{4}L_A+U_A$. Here and in the sequel, $D_A,-L_A$ and $-U_A$ are the diagonal part, the strictly lower-triangular and the strictly upper-triangular parts of $A$, respectively.

\noindent\textbf{RMS}: the relaxed-based matrix splitting iteration method \cite{soso2023}
\begin{equation*}
\begin{cases}
x^{(k+1)}=S^{-1}(Tx^{(k)}+By^{(k)}+b),\\
y^{(k+1)}=(1-\tau)y^{(k)}+\tau |x^{(k+1)}|,
\end{cases}
\end{equation*}
where $S=M$ and $T=N$.

\noindent\textbf{FPI}: the fixed point iteration method \cite{lild2023}
\begin{equation*}
\begin{cases}
x^{(k+1)}=A^{-1}(By^{(k)}+b),\\
y^{(k+1)}=(1-\tau)y^{(k)}+\tau |x^{(k+1)}|.
\end{cases}
\end{equation*}

\noindent\textbf{GN}: the generalized Newton method \cite{mang2009}
\begin{equation*}
x^{(k+1)}=(A-B\mathcal{D}(x^{(k)}))^{-1}b.
\end{equation*}

\noindent\textbf{Picard}: the Picard iteration method \cite{rohf2014}
\begin{equation*}
x^{(k+1)}=A^{-1}(B|x^{(k)}|+b).
\end{equation*}

\noindent\textbf{NSNA}: the non-monotone smoothing Newton algorithm\cite{cyhm2025} with the same parameters used in \cite{cyhm2025}.\\

\noindent\textbf{MN}: the \Cref{alg:mn} \cite{wacc2019}.\\

\noindent\textbf{SSMN}: the shift splitting MN iteration method \cite{liyi2021}
\begin{equation*}
x^{(k+1)}=(A+\Omega)^{-1}((\Omega-A) x^{(k)}+2B|x^{(k)}|+2b).
\end{equation*}

\noindent\textbf{MNMS}: the modified Newton-based matrix splitting iteration method \cite{zcss2024}
\begin{equation*}
x^{(k+1)}=[\Omega+M_1-M_2\mathcal{D}(x^{(k)})]^{-1}(\Omega x^{(k)}+N_1x^{(k)}-N_2|x^{(k)}|+b),
\end{equation*}
where $M_1=M,N_1=N,M_2=D_B-\frac{1}{4}L_B$ and $N_2=\frac{3}{4}L_B+U_B$ with $D_B,-L_B$ and $-U_B$ are the diagonal part, the strictly lower-triangular and the strictly upper-triangular parts of $B$, respectively.

For the methods MN, SSMN and MNMS, $\Omega=0.5D_A$. For the GNMS, RMS and FPI methods, $\tau_{opt}$ is the experimental optimal parameter within the range $[0:0.01:2]$, and CPU$_{opt}$ refers to the time taken to find the optimal parameter.

\begin{exam}\label{exam:5.3}
Let $m$ be a positive integer and $n=m^2$. Consider the GAVEs \eqref{eq:gave}, in which $A, B \in\mathbb{R}^{n\times n}$ are given by
\begin{equation*}
A=\begin{cases}
      20 + 10 * rand(n), & \mbox{if } i = j \\
      -0.001 * rand(n), & \mbox{otherwise}
    \end{cases}~~
B=\begin{cases}
      4 * rand(n), & \mbox{if } i = j \\
      -0.001 * rand(n), & \mbox{otherwise}.
    \end{cases}
\end{equation*}
To ensure reproducibility, we fixed the random seed to $42$ before generating the random matrices, using rng($42$) in MATLAB. In addition, to investigate the behavior of the GGS with singular matrix $B$, we set B(end,:)=B(end-1, :). All iterative methods start with the initial point $x^{(0)}=(0,0,0,\ldots)^\top\in\mathbb{R}^n$ and $y^{(0)}=b$ (if necessary). The iteration is terminated if $\text{RES}(z^{(k)})\leq 10^{-8}$ or $\text{IT}\geq 100$. Here, `RES' is set to be
$$
\text{RES}:=\frac{\|Ax^{(k)}-B|x^{(k)}|-b\|_2}{\|b\|_2}.
$$

We examine this convergence condition \eqref{eq:con} by varying the $m$ in the range $[60:10:100]$. To this end, let
$$
{\rm inf}\_{\rm norm} = \|(D_A-L_A)^{-1}U_A\|_\infty+\|(D_A-L_A)^{-1}U_B\|_\infty + \|(D_A-L_A)^{-1}(D_B+|L_B|)\|_\infty.
$$
As shown in the Figure \ref{fig3}, the condition \eqref{eq:con} is satisfied.

\begin{figure}[htbp]
  \centering
  \includegraphics[width=0.70\linewidth]{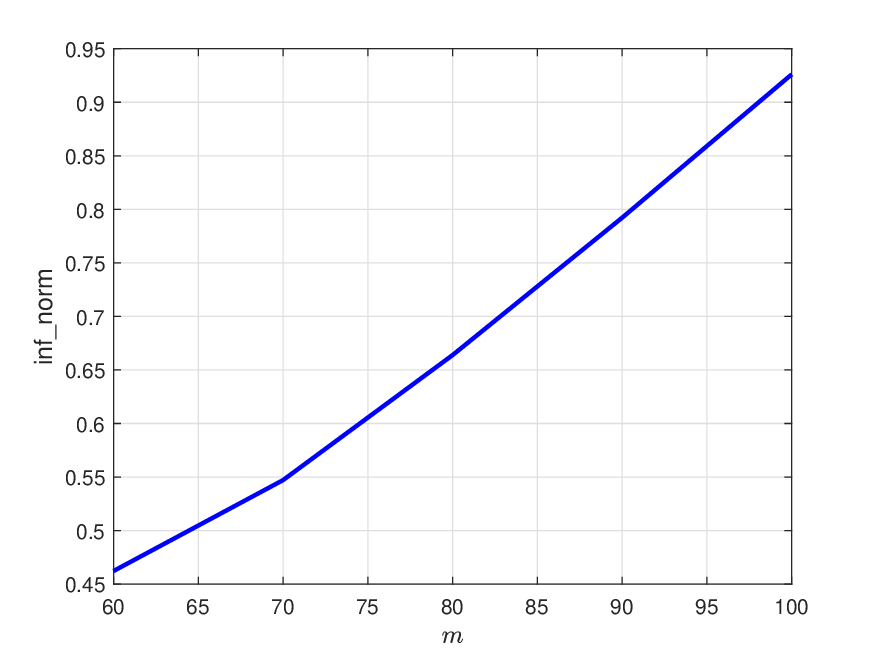}
  \caption{Relationship between $m$ and {\rm inf}\_{\rm norm} in Example \ref{exam:5.3}}\label{fig3}
\end{figure}

\begin{table}[htp]
\centering
\footnotesize
\caption{Numerical comparison of the testing methods in Example \ref{exam:5.3}.}\label{table5.3}
\begin{tabular}{cccccccccc}\hline
&Method   &  &$m$                     \\\cline{3-8}
&{}         &   &60         &70        &80         &90      &100     \\\hline
&GGS \\
& &IT      &3               &3            &3           &3           &3\\
& &CPU   &\textbf{0.0914}  &\textbf{0.1607}  &\textbf{0.2364} &\textbf{0.3255}   &\textbf{0.4959}\\
& &RES   &2.5601e-10  &4.0867e-10  &6.1429e-10  &8.8826e-10  &1.1878e-09\\
&GNMS \\
& &$\tau_{opt}$  &1.05  &1.05  &1.04  &1.05  &1.05\\
& &IT                   &9       &9      &9       &9       &9\\
& &CPU$_{opt}$  &2.5925  &4.5085   &6.7955  &10.9341  &16.3335\\
& &CPU                &0.1861  &0.3402  &0.5938  &0.9791  &1.5486\\
& &RES     &6.0314e-09  &5.2380e-09 &8.9340e-09  &6.6178e-09 &9.1788e-09\\
&RMS \\
& &$\tau_{opt}$  &1.07  &1.07  &1.02  &1.05  &1.01\\
& &IT                    &9    &9        &10       &10      &11    \\
& &CPU$_{opt}$   &2.3369 &4.0083  &6.8234  &11.0313 &16.4864\\
& &CPU                 &0.1788 &0.3192  &0.6091  &1.0120 &1.6424\\
& &RES      &8.4239e-09  &7.5739e-09  &7.4656e-09  &6.9346e-09  &9.6828e-09\\
&FPI \\
& &$\tau_{opt}$  &1.07  &1.03  &1.01  &1.04  &0.95\\
& &IT                    &9      &10      &10    &10      &12    \\
& &CPU$_{opt}$   &1.9457 &2.8848  &4.8055  &8.3914 &11.7268\\
& &CPU                 &0.1501 &0.2904  &0.4879  &0.8263 &1.4465\\
& &RES      &8.2283e-09  &7.3230e-09  &6.5618e-09  &7.8190e-09  &9.1991e-09\\
&GN \\
& &IT       &2               &2            &2           &2           &2\\
& &CPU    &0.6275       &1.1513   &2.3611  &4.2372   &8.1908\\
& &RES   &3.3166e-15 &3.7873e-15  &4.2747e-15  &4.8541e-15  &5.4069e-15\\
&Picard \\
& &IT      &10               &10                &10           &10              &11\\
& &CPU   &2.7168         &5.5041         &10.5166  &20.2776      &43.7244\\
& &RES   &2.5125e-09  &2.5369e-09  &2.2125e-09  &1.8380e-09  &2.0952e-09\\
&NSNA \\
& &IT      &3                &3                   &3                 &3                &3\\
& &CPU   &3.6507        &9.4282          &20.1135       &42.2720     &86.2828\\
& &RES   &8.1475e-14   &1.6361e-13   &3.4374e-13   &7.5085e-13 &1.5340e-12 \\
&MN \\
& &IT      &21              &21          &21         &21          &21\\
& &CPU   &0.6771       &1.2285    &2.0342  &3.3699  &4.9941\\
& &RES   &6.5730e-09 &6.3702e-098  &5.6349e-09  &4.7618e-09 &6.5094e-09\\
&SSMN \\
& &IT      &17               &18                 &19             &21            &23\\
& &CPU   &0.9526        &1.7362           &3.3269     &5.6516   &9.5320\\
& &RES   &8.3287e-09  &6.1207e-09    &6.3519e-09  &4.4148e-09  &5.4101e-09\\
&MNMS \\
& &IT      &18               &18                 &18                &18          &18\\
& &CPU   &1.9426        &3.4567           &5.9812         &9.7676    &20.2350\\
& &RES   &7.7147e-09  &7.3095e-09   &6.7733e-09   &6.1815e-09  &5.9773e-09\\
\hline
\end{tabular}
\end{table}

Numerical results are shown in Table \ref{table5.3}, from which we can see that all tested iterative methods are convergent, and the GGS method consumes the least CPU time. Although the GN iterative method requires fewer iteration steps than the GGS iterative method, it consumes more CPU time. This is mainly because the GN iterative method needs to solve different linear equation systems at each iteration step. For the parameterized methods GNMS, RMS, and FPI methods, they incur additional computational costs in the process of identifying optimal parameters. As the dimension $m$ increases, the task of identifying optimal parameters becomes more challenging.
\end{exam}

\section{Conclusion}\label{sec:conclusion}
The paper focuses on solving the generalized absolute value equations, using a Generalized Gauss-Seidel (GGS) iteration method. The GGS method is characterized by its parameter-free strategy, which eliminates the need for parameter tuning. Under some mild conditions, the convergence of the GGS method is analyzed. Numerical results demonstrate our claims.


\begin{thebibliography}{99}

\bibitem{acha2018}
M. Achache, N. Hazzam. Solving absolute value equations via complementarity and interior-point methods, \emph{J. Nonlinear Funct. Anal.}, 2018, 2018: 39.


\bibitem{cyhm2025}
C.-R. Chen, D.-M. Yu, D.-R. Han, C.-F. Ma.  A non-monotone smoothing Newton algorithm for solving the system of generalized absolute value equations, \emph{J. Comp. Math.}, 2025, 43: 438--460.

\bibitem{cops2009}
R. W. Cottle, J. S. Pang, R. E. Stone. The Linear Complementarity Problem, \emph{SIAM}, Philadelphia, 2009.

\bibitem{dawz2024}
Y.-H. Dai, J.-N. Wang, and L.-W. Zhang.
 Optimality conditions and numerical algorithms for a class of
  linearly constrained minimax optimization problems,
 {\em SIAM J. Optim.}, 34: 2883--2916, 2024.

\bibitem{edhs2017}
V. Edalatpour, D. Hezari, D. K. Salkuyeh. A generalization of the Gauss-Seidel iteration method for solving absolute value equations, \emph{Appl. Math. Comput.}, 2017, 293: 156--167.

\bibitem{frma1989}
A. Frommer, G. Mayer. Convergence of relaxed parallel multisplitting methods, \emph{Linear Algebra Appl.}, 1989, 119: 141--152.

\bibitem{huhz2011}
S.-L. Hu, Z.-H. Huang, Q. Zhang. A generalized Newton method for absolute value equations associated with second order cones, \emph{J. Comput. Appl. Math.}, 2011, 235: 1490--1501.

\bibitem{jizh2013}
X.-Q. Jiang, Y. Zhang. A smoothging-type algorithm for absolute value equations, \emph{J. Ind. Manag. Optim.}, 2013, 9: 789--798,.


\bibitem{kdhm2024}
S. Kumar, Deepmala, M. Hladík, H. Moosaei. Characterization of unique solvability of absolute value equations: an overview, extensions, and future directions, \emph{Optim. Lett.}, 2024, 18: 889--907.

\bibitem{lild2023}
X. Li, Y.-X. Li, Y. Dou. Shift-splitting fixed point iteration method for solving generalized absolute value equations, \emph{Numer. Algorithms}, 2023, 93:695--710.

\bibitem{liyi2021}
X. Li, X.-X. Yin. A new modified Newton-type iteration method for solving generalized absolute value equations, arXiv preprint, arXiv:2103.09452,  2021.

\bibitem{lich2025}
X.-H. Li, C.-R. Chen. An efficient Newton-type matrix splitting algorithm for solving generalized absolute value equations with application to ridge regression problems, \emph{J. Comput. Appl. Math.}, 2025, 457: 116329.

\bibitem{lilw2018}
Y.-Y. Lian, C.-X. Li, S.-L. Wu. Weaker convergent results of the generalized Newton method for the generalized absolute value equations, \emph{J. Comput. Appl. Math.}, 2018, 338: 221--226.

\bibitem{love2013}
T. Lotfi, H. Veiseh. A note on unique solvability of the asolute value equation, \emph{J. Linear Topol. Algeb.}, 2013, 2: 77--81.

\bibitem{mang2007}
O. L. Mangasarian. Absolute value programming, \emph{Comput. Optim. Appl.}, 2007, 36: 43--53.

\bibitem{mang2009}
O.L. Mangasarian. A generalized Newton method for absolute value equations, \emph{Optim. Lett.}, 2009, 3: 101--108.

\bibitem{mame2006}
O. L. Mangasarian, R. R. Meyer. Absolute value equations, \emph{Linear Algebra Appl.}, 2006, 419: 359--367.

\bibitem{mezz2020}
F. Mezzadri. On the solution of general absolute value equation, \emph{Appl. Math. Lett.}, 2020, 107: 106462.

\bibitem{pork2009}
O. Prokopyev. On equivalent reformulations for absolute value equations, \emph{Comput. Optim. Appl.}, 2009, 44:363--372.

\bibitem{rohn2004}
J. Rohn. A theorem of the alternatives for the equation $Ax+ B|x|= b$, \emph{Linear Multilinear Algebra}, 2004, 52: 421--426.

\bibitem{rohn2009}
J. Rohn. On unique solvability of the absolute value equation, \emph{Optim. Lett.}, 2009, 3: 603--606.

\bibitem{rohf2014}
J. Rohn, V. Hooshyarbakhsh, R. Farhadsefat. An iterative method for solving absolute value equations and sufficient conditions for unique solvability, \emph{Optim. Lett.}, 2014, 8: 35--44.

\bibitem{saad2003}
Y. Saad. Iterative Methods for Sparse Linear Systems (Second edition), \emph{SIAM}, Philadelphia, 2003.

\bibitem{schn1984}
H. Schneider. Theorems on $M$-splittings of a singular $M$-matrix which depend on graph structure, \emph{Linear Algebra Appl.}, 1984, 58: 407--424.

\bibitem{soso2023}
J. Song, Y.-Z. Song. Relaxed-based matrix splitting methods for solving absolute value equations, \emph{Comput. Appl. Math.}, 2023, 42: 19.

\bibitem{tazh2019}
J.-Y. Tang, J.-C. Zhou. A quadratically convergent descent method for the absolute value equation $Ax + B|x| = b$, \emph{Oper. Res. Lett.}, 2019, 47: 229--234.

\bibitem{varg1962}
R.S. Varga. Matrix Iterative Analysis, \emph{Prentice-Hall}, Englewood Cliffs, 1962.

\bibitem{wacc2019}
A. Wang, Y. Cao, J.-X. Chen. Modified Newton-type iteration methods for generalized absolute value equations, \emph{J. Optim. Theory Appl.}, 2019, 181: 216--230.

\bibitem{wuli2020}
S.-L. Wu, C.-X. Li. A note on unique solvability of the absolute value equation, \emph{Optim. Lett.}, 2020, 14: 1957--1960.

\bibitem{wush2021}
S.-L. Wu, S.-Q. Shen. On the unique solution of the generalized absolute value equation, \emph{Optim. Lett.}, 2021, 15: 2017--2024.

\bibitem{xiqh2024}
J.-X. Xie, H.-D. Qi, and D.-R. Han. Randomized iterative methods for generalized absolute value equations: Solvability and error bounds, {\em arXiv preprint}, arXiv:2405.04091, 2024.

\bibitem{zhyi2013}
N. Zheng, J.-F. Yin. Accelerated modulus-based matrix splitting iteration methods for linear complementarity problem, \emph{Numer. Algorithms}, 2013, 64: 245--262.

\bibitem{zcss2024}
C.-C. Zhou, Y. Cao, Q.-Q. Shen, Q. Shi. A modified Newton-based matrix splitting iteration  method for generalized absolute value equations, \emph{J. Comput. Appl. Math.}, 2024, 442: 115747.

\bibitem{zhwl2021}
H.-Y. Zhou, S.-L. Wu, C.-X. Li. Newton-based matrix splitting method for generalized absolute value equation, \emph{J. Comput. Appl. Math.}, 2021, 394: 113578.

\end{thebibliography}
\end{document}